\RequirePackage{fix-cm}
\documentclass[smallextended]{svjour3} 
\smartqed

\makeatletter
\@addtoreset{equation}{section}
\makeatother

\usepackage{amssymb,amsmath,amsfonts,latexsym}
\usepackage{mathtools}
\usepackage{multirow,array}
\usepackage{graphicx}
\usepackage{subfigure}
\usepackage{epstopdf}
\usepackage{float}
\usepackage{color,xcolor}
\usepackage{booktabs}
\usepackage{mathrsfs}
\usepackage{placeins}
\allowdisplaybreaks
\usepackage[hidelinks]{hyperref}
\graphicspath{{./}{figures/}{results_LLF_Caputo_article/}{results_LLF_FDE_article/}}

\def\proofend{\hfill$\Box$}

\journalname{Fract. Calc. Appl. Anal.}

\begin{document}

\title{An Exact-Moment Local Legendre Frame Method with Block Convolution for Caputo Fractional Differentiation}

\titlerunning{Exact-moment LLF method for Caputo fractional differentiation}

\author{Zhenyu Zhao$^{1,*}$
\and Benxue Gong$^1$
\and Tinggang Zhao$^1$
\and Xianzheng Jia$^1$}

\authorrunning{Z. Zhao et al.}

\institute{Zhenyu Zhao$^{1,*}$
\at School of Mathematics and Statistics, Shandong University of Technology, Zibo, 255049, China \\
\email{dongdongzb@gmail.com} $^*$ corresponding author
\and Benxue Gong$^{1}$
\at School of Mathematics and Statistics, Shandong University of Technology, Zibo, 255049, China
\and Tinggang Zhao$^{1}$
\at School of Mathematics and Statistics, Shandong University of Technology, Zibo, 255049, China
\and Xianzheng Jia$^{1}$
\at School of Mathematics and Statistics, Shandong University of Technology, Zibo, 255049, China
}

\date{}

\maketitle

\begin{abstract}
We propose a local Legendre frame method for the accurate computation of
Caputo fractional derivatives of order \(0<\alpha<1\). On each local
subinterval, the function is represented by a restricted Legendre frame
obtained from scaled Legendre polynomials on an extended interval. The local
coefficients are computed from equispaced samples by an exponentially weighted
GTSVD regularization. The Caputo derivative is then evaluated by applying the
weakly singular fractional integral to the derivatives of the local frame
basis functions. Since these derivatives are polynomials, the corresponding
Caputo weights can be written in terms of finite weighted moments, so that the
singular kernel is treated analytically rather than by a low-order quadrature
rule. For uniform partitions, the history weights have a block-dependent
structure and can be reused efficiently. The error analysis separates the
local frame reconstruction from the Caputo integration. In particular, the
Caputo error is bounded by the derivative reconstruction error, while the
latter is obtained from the \(L^2\) reconstruction error and a weighted
smoothness bound of the GTSVD approximation through an interpolation argument.
For analytic local functions with exponential coefficient decay, this leads
to exponential-type convergence of the derivative and hence of the Caputo
approximation. Numerical experiments confirm the accuracy of the exact moment
weights, the effectiveness of the local weighted reconstruction, and the
efficiency of the block implementation.


\end{abstract}

\section{Introduction}
\label{sec:introduction}

Fractional calculus provides a useful mathematical framework for modeling
memory, hereditary effects, and nonlocal interactions
\cite{Podlubny1999,Diethelm2010,LiLiu2018,Luchko2020}. Fractional differential
operators have been used in anomalous diffusion, viscoelasticity, relaxation
processes, porous media transport, and related applications
\cite{Podlubny1999,Diethelm2010,YanSunZhang2017}. Among
the commonly used fractional derivatives, the Caputo derivative is especially
convenient for initial value problems because it is compatible with classical
initial conditions and the derivative of a constant vanishes
\cite{Podlubny1999,Diethelm2010}. For
\(0<\alpha<1\), the left-sided Caputo derivative is
\begin{equation}
{}^{C}D_{0+}^{\alpha}f(x)=
\frac{1}{\Gamma(1-\alpha)}\int_0^x (x-s)^{-\alpha} f'(s)\,ds .
\label{eq:caputo_def_intro}
\end{equation}
The weakly singular kernel and the dependence on the whole history make the
accurate and efficient computation of \eqref{eq:caputo_def_intro} more delicate
than classical integer-order differentiation.

Many numerical methods have been developed for Caputo derivatives. Classical
finite difference formulas, such as the L1 approximation and its higher-order
variants, are especially useful in time-fractional evolution problems because
they lead to causal step-by-step schemes
\cite{Lubich1986,Alikhanov2015,YanSunZhang2017,MokhtariMostajeran2019}.
At each time level, the current unknown is coupled with already computed
history values. This triangular structure explains why low-order time-stepping
schemes remain important in practical simulations. However, direct history
summation requires \(O(M^2)\) work for \(M\) time levels, and fast history
algorithms based on kernel compression or sum-of-exponentials approximations
are often needed for long-time computations
\cite{YanSunZhang2017}. Moreover, such schemes have
algebraic convergence rates and may require many time levels when high accuracy
is desired for smooth data.

Another important class of high-order methods is based on global spectral
approximation. Legendre, Chebyshev, and Jacobi spectral methods have been used
for fractional integrals, Caputo derivatives, and fractional differential
equations \cite{LiZengLiu2012}. For analytic functions, global spectral methods
can achieve very high accuracy. However, their accuracy usually relies on a
global approximation over the entire interval, often with nonuniform spectral
nodes. This global structure is less convenient when data are naturally sampled
on equispaced grids, when localized or piecewise smooth structures are present,
or when one wants a blockwise time-marching formulation rather than an
all-at-once global-in-time discretization.

Frame approximation offers a flexible way to construct stable approximations
from redundant or nearly redundant systems
\cite{AdcockHuybrechs2019,AdcockHuybrechs2020}. Fourier extension is a typical
example of frame approximation for nonperiodic functions
\cite{Boyd2002,Huybrechs2010,MatthysenHuybrechs2016,MatthysenHuybrechs2018}.
Local Fourier extension and local Legendre frame methods further combine the
frame idea with interval subdivision and reusable local discretizations
\cite{GongZhaoWang2026}. In the present Caputo setting, we choose a local
Legendre frame rather than a Fourier extension frame because the derivatives of
the Legendre frame functions are polynomials. Therefore the Caputo action on
these derivatives can be written in terms of finite algebraic moments, which is
more direct than treating oscillatory Fourier basis functions inside a weakly
singular fractional integral.

The least-squares systems associated with frame approximations are typically
ill-conditioned, and regularization is required. Weighted generalized inverses
and GSVD-type regularizations provide a mechanism for imposing structural
information on the recovered coefficients
\cite{Hansen1989,Hansen1992,ZhaoWangLi2026}. This is relevant for Caputo
differentiation because the operator in \eqref{eq:caputo_def_intro} acts on the
first derivative of \(f\), and hence the stability of the integer-order
derivative reconstruction directly affects the fractional derivative.

Motivated by these observations, we develop an exact-moment local Legendre
frame (LLF) method for Caputo fractional differentiation from local equispaced
data. The method separates the computation into two stages. First, a stable
local Legendre frame representation is recovered by an exponentially weighted
generalized inverse. Second, the Caputo operator is applied exactly to the
derivatives of the local frame functions. Since these derivatives are
polynomials, the weakly singular integral is reduced to explicit weighted
moments. For a uniform partition, the history weights depend only on relative
block positions, which yields a block-convolution implementation.

The main contributions of this paper are summarized as follows.
\begin{itemize}
\item We formulate a Caputo differentiation method based on the restricted
scaled Legendre frame. The sampling points are local equispaced points on the
reference interval, and all physical subintervals share the same local frame
matrix.
\item We derive exact moment formulas for the Caputo action on derivatives of
local frame functions. Thus the weak singularity is incorporated into
precomputed weights, without a separate low-order quadrature of the singular
kernel.
\item We show that the Caputo error is controlled by the local derivative
reconstruction error. The analysis is stated in terms of the actual weighted
GTSVD reconstruction, not in terms of an ordinary Legendre projection.
\item We describe the computational complexity in terms of the total number of
sampling points. For fixed local parameters, the storage of history weights is
linear in the number of samples, and FFT-based history accumulation is nearly
linear up to a logarithmic factor.
\item Numerical experiments include a closed-form verification of the exact
moment weights, comparisons with the classical L1 scheme and a global
Chebyshev spectral benchmark, and a manufactured fractional differential
equation. The results demonstrate the accuracy of the moment construction,
high accuracy from local equispaced data, and the blockwise structure of the
method.
\end{itemize}

The remainder of this paper is organized as follows. Section~\ref{sec:local_llf}
introduces the local Legendre frame approximation and the exponentially weighted
generalized inverse. Section~\ref{sec:caputo_llf} derives the exact-moment
Caputo weights and the block-convolution implementation. Section~\ref{sec:error_analysis}
presents the error analysis. Section~\ref{sec:numerical_experiments} reports
numerical experiments. Section~\ref{sec:conclusion} concludes the paper.

\section{Local Legendre Frame Approximation with Exponentially Weighted Generalized Inverse}
\label{sec:local_llf}

This section recalls the local Legendre frame structure used in the paper and
introduces the exponentially weighted generalized inverse. The important point
is that the method does not use the ordinary Legendre basis on \([-1,1]\) as an
orthogonal projection basis. Instead, it uses Legendre polynomials orthonormal
on a larger interval and restricts them to the physical reference interval. The
resulting system is a polynomial frame on the smaller interval.

\subsection{Restricted scaled Legendre frame}
\label{subsec:restricted_frame}

Let
\[
\Lambda=[-1,1],\qquad \Lambda_e=[-T,T],\qquad T>1 .
\]
Let \(\{p_j\}_{j\ge 0}\) be the orthonormal Legendre basis on \([-1,1]\), and
define the scaled Legendre functions on the extended interval by
\begin{equation}
P_j^{(T)}(t)=\frac{1}{\sqrt T}p_j\left(\frac{t}{T}\right),
\qquad t\in[-T,T],\qquad j=0,1,\ldots .
\label{eq:scaled_legendre_frame_basis}
\end{equation}
Then \(\{P_j^{(T)}\}_{j\ge0}\) is orthonormal in \(L^2([-T,T])\). In this work,
however, the approximation is evaluated and fitted only on \(\Lambda=[-1,1]\).
Thus the finite system
\[
\Phi_{N_b}^{(T)}=\{P_j^{(T)}|_{\Lambda}:j=0,\ldots,N_b-1\}
\]
is a restricted polynomial frame on \(\Lambda\), rather than an orthogonal basis
on \(\Lambda\). Here \(N_b\) denotes the number of local frame modes.

For a coefficient vector \(c=(c_0,\ldots,c_{N_b-1})^T\), define the local
synthesis operator
\begin{equation}
L_{N_b}^{(T)}c(t)=\sum_{j=0}^{N_b-1}c_jP_j^{(T)}(t),
\qquad t\in[-1,1].
\label{eq:local_synthesis_operator}
\end{equation}
The restriction to \([-1,1]\) is essential: all local data are sampled on this
inner interval, while the orthogonality of the underlying functions comes from
the extended interval \([-T,T]\).

\subsection{Local rescaling and discrete data}
\label{subsec:local_rescaling_data}

Let \([0,b]\) be partitioned as
\[
0=x_0<x_1<\cdots<x_K=b,
\qquad I_k=[x_{k-1},x_k].
\]
For each subinterval, set
\[
c_k=\frac{x_{k-1}+x_k}{2},\qquad h_k=\frac{x_k-x_{k-1}}{2},
\qquad x=c_k+h_kt,\quad t\in[-1,1].
\]
The local function is
\[
g_k(t)=f(c_k+h_kt),\qquad t\in[-1,1].
\]
Let \(m\) be the number of local equispaced samples and define
\[
t_r=-1+\frac{2r}{m-1},\qquad r=0,\ldots,m-1.
\]
The physical sampling points are
\[
x_{k,r}=c_k+h_kt_r,
\qquad r=0,\ldots,m-1,
\]
and the local data vector is
\[
\mathbf g_k=(g_k(t_0),g_k(t_1),\ldots,g_k(t_{m-1}))^T.
\]
The normalized discrete frame matrix is
\begin{equation}
(A_{m,N_b}^{(T)})_{rj}=\frac{1}{\sqrt m}P_j^{(T)}(t_r),
\qquad r=0,\ldots,m-1,
\quad j=0,\ldots,N_b-1 .
\label{eq:discrete_frame_matrix}
\end{equation}
The coefficients on each subinterval are computed from
\begin{equation}
A_{m,N_b}^{(T)}c_k\approx \frac{1}{\sqrt m}\mathbf g_k .
\label{eq:local_discrete_ls}
\end{equation}
Since the same reference nodes and the same restricted frame are used on every
subinterval, the matrix \(A_{m,N_b}^{(T)}\) is independent of \(k\) and can be
factorized once.

\subsection{Exponentially weighted generalized inverse}
\label{subsec:weighted_ginv}

Although the data are exact in the present paper, direct inversion of the local
frame system may be unfavorable for derivative reconstruction. We therefore use
an exponentially weighted generalized inverse. Let
\[
R_\beta=\operatorname{diag}(r_0,r_1,\ldots,r_{N_b-1}),
\qquad r_j=e^{\beta j},\qquad \beta\ge0 .
\]
Introduce the weighted variable \(c=R_\beta^{-1}z\) and define
\[
G_\beta=A_{m,N_b}^{(T)}R_\beta^{-1}.
\]
Let
\[
G_\beta=U\Sigma V^T,
\qquad \Sigma=\operatorname{diag}(\sigma_0,\sigma_1,\ldots)
\]
be its singular value decomposition. With truncation threshold \(\eta>0\), the
computed coefficient vector is
\begin{equation}
c_{k,\eta}
=
R_\beta^{-1}
\sum_{\sigma_i>\eta}
\frac{\left\langle m^{-1/2}\mathbf g_k,u_i\right\rangle}{\sigma_i}v_i .
\label{eq:weighted_gtsvd_coeff}
\end{equation}
The local LLF approximation on \(I_k\) is therefore
\begin{equation}
p_{k,N_b}^{\eta}(x)
=
\sum_{j=0}^{N_b-1}c_{k,\eta,j}
P_j^{(T)}\left(\frac{x-c_k}{h_k}\right),
\qquad x\in I_k .
\label{eq:local_physical_llf_approx}
\end{equation}
Its derivative is
\begin{equation}
q_{N_b}^{\eta}(x)
=
\frac{d}{dx}p_{k,N_b}^{\eta}(x)
=
\frac{1}{h_k}
\sum_{j=0}^{N_b-1}c_{k,\eta,j}
\frac{d}{dt}P_j^{(T)}(t)
\bigg|_{t=(x-c_k)/h_k},
\qquad x\in I_k .
\label{eq:local_derivative_q}
\end{equation}
This is the derivative approximation used in the Caputo formula.

\begin{remark}
The coefficients \(c_{k,\eta}\) in \eqref{eq:weighted_gtsvd_coeff} are not the
ordinary Legendre projection coefficients of \(g_k\) on \([-1,1]\). They are the
regularized coefficients produced by the weighted generalized inverse of the
restricted frame matrix. Consequently, the analysis below is formulated in
terms of the actual reconstructed derivative \(q_{N_b}^{\eta}\), not in terms of
ordinary Legendre projection errors.
\end{remark}
\begin{remark}
In exact-data experiments with sufficiently smooth local functions, the
unweighted truncated inverse may already give nearly machine-precision
derivative values, because the relevant data components are mostly contained
in the stable singular subspace. In this regime, the numerical difference
between the unweighted and exponentially weighted reconstructions can be
small. The role of the exponential weight is therefore not merely to improve
the accuracy in all exact-data tests, but to select a smoother coefficient
vector in the ill-conditioned frame system. This becomes important for
derivative reconstruction when perturbations, roundoff errors, or unresolved
high-index components are present. For this reason, the present paper uses
the weighted formulation as a stability mechanism and as the framework under
which derivative and Caputo error estimates are obtained.
\end{remark}
\begin{remark}[Effect of localization]
The affine change of variables \(x=c_k+h_kt\) can improve local coefficient
decay. If \(f\) is analytic near \(I_k\) and the nearest complex singularity has
a distance of order \(d_k\) from the physical interval, then the mapped local
function \(g_k(t)=f(c_k+h_kt)\) sees a distance of order \(d_k/h_k\) in the
\(t\)-plane. Thus local Legendre-type coefficients may decay faster as the
subinterval size decreases. For oscillatory functions, for instance,
\(e^{i\omega x}\) becomes \(e^{i\omega c_k}e^{i(\omega h_k)t}\), so the effective
local frequency is \(\omega h_k\). This explains why subdivision improves the
weighted source behavior, although the GTSVD coefficients are still not
ordinary projection coefficients.
\end{remark}

\section{Caputo Differentiation Based on the Local Legendre Frame}
\label{sec:caputo_llf}

We now derive the exact-moment Caputo differentiation formula based on the
local derivative approximation \eqref{eq:local_derivative_q}. Throughout this
section, \(0<\alpha<1\).

\subsection{Local Caputo weights}
\label{subsec:local_caputo_weights}

For an evaluation point \(x_i\in(0,b]\), the proposed approximation is
\begin{equation}
\mathcal D_{N_b,\eta}^{\alpha}f(x_i)
=
\frac{1}{\Gamma(1-\alpha)}
\int_0^{x_i}(x_i-\xi)^{-\alpha}q_{N_b}^{\eta}(\xi)\,d\xi .
\label{eq:caputo_approx_q}
\end{equation}
Let \(I_k\) be a source interval contributing to this integral. Under the
change of variables \(\xi=c_k+h_kt\), the relevant part of the integral is
\[
\frac{h_k^{-\alpha}}{\Gamma(1-\alpha)}
\int_{t_L}^{t_U}
(\lambda_{i,k}-t)^{-\alpha}
\sum_{j=0}^{N_b-1}c_{k,\eta,j}(P_j^{(T)})'(t)\,dt,
\]
where
\[
\lambda_{i,k}=\frac{x_i-c_k}{h_k},
\]
and \([t_L,t_U]\subset[-1,1]\) is either a full source interval or a partial
current interval. Hence the local weight for mode \(j\) is
\begin{equation}
W_{i,k,j}^{(\alpha)}
=
\frac{h_k^{-\alpha}}{\Gamma(1-\alpha)}
\int_{t_L}^{t_U}
(\lambda_{i,k}-t)^{-\alpha}(P_j^{(T)})'(t)\,dt .
\label{eq:local_caputo_weight}
\end{equation}
Then
\[
\mathcal D_{N_b,\eta}^{\alpha}f(x_i)
=
\sum_{k}\sum_{j=0}^{N_b-1}W_{i,k,j}^{(\alpha)}c_{k,\eta,j},
\]
where the outer sum is over all local intervals lying in the history of
\(x_i\), including the partial interval containing \(x_i\).

\subsection{Exact moment representation}
\label{subsec:exact_moments}

Since \(P_j^{(T)}\) is a polynomial in \(t\), its derivative has the finite form
\[
(P_j^{(T)})'(t)=\sum_{r=0}^{j-1}d_{jr}^{(T)}t^r .
\]
Therefore the weight \eqref{eq:local_caputo_weight} can be written as
\begin{equation}
W_{i,k,j}^{(\alpha)}
=
\frac{h_k^{-\alpha}}{\Gamma(1-\alpha)}
\sum_{r=0}^{j-1}d_{jr}^{(T)}
J_r^{(\alpha)}(t_L,t_U;\lambda_{i,k}),
\label{eq:weight_moment_sum}
\end{equation}
where
\begin{equation}
J_r^{(\alpha)}(t_L,t_U;\lambda)
=
\int_{t_L}^{t_U}(\lambda-t)^{-\alpha}t^r\,dt .
\label{eq:basic_moment}
\end{equation}
These moments are finite and can be evaluated analytically. A stable expression
near the singular endpoint is obtained by setting
\[
\delta=\lambda-t_U,
\qquad L=t_U-t_L,
\qquad t=t_U-y .
\]
Then
\begin{equation}
J_r^{(\alpha)}(t_L,t_U;\lambda)
=
\sum_{s=0}^{r}\binom{r}{s}t_U^{r-s}(-1)^sH_s(\delta,L),
\label{eq:J_via_H}
\end{equation}
with
\[
H_s(\delta,L)=\int_0^L y^s(\delta+y)^{-\alpha}\,dy .
\]
If \(\delta=0\), then
\[
H_s(0,L)=\frac{L^{s+1-\alpha}}{s+1-\alpha}.
\]
If \(\delta>0\), then
\[
H_s(\delta,L)=
\sum_{r=0}^{s}\binom{s}{r}(-\delta)^{s-r}
\frac{(\delta+L)^{r+1-\alpha}-\delta^{r+1-\alpha}}{r+1-\alpha}.
\]
Thus the weak singularity is treated exactly at the level of local moments.

\subsection{Block-convolution form and complexity}
\label{subsec:complexity}

We now make explicit how the block-convolution structure arises from the
uniform partition. Assume that
\[
0=x_0<x_1<\cdots<x_K=b,\qquad H=x_k-x_{k-1}=\frac{b}{K},
\]
and set
\[
c_k=\frac{x_{k-1}+x_k}{2},\qquad h=\frac{H}{2}.
\]
On each block \(I_k=[x_{k-1},x_k]\), the local LLF approximation has
coefficients \(c_{k,j}\), \(j=0,\ldots,N_b-1\). We evaluate the Caputo
derivative at the same relative positions in each block:
\[
x_{k,\mu}=x_{k-1}+\rho_\mu H,\qquad
0<\rho_\mu\le 1,\qquad \mu=1,\ldots,s.
\]
It is convenient to introduce
\[
\theta_\mu=2\rho_\mu-1\in[-1,1].
\]
Thus \(x_{k,\mu}=c_k+h\theta_\mu\).

For the current block \(I_k\), only the part \([x_{k-1},x_{k,\mu}]\) contributes
to the Caputo integral. In the local variable \(t=(\xi-c_k)/h\), this gives the
current-block weights
\[
B_{\mu,j}^{(\alpha)}
=
\frac{h^{-\alpha}}{\Gamma(1-\alpha)}
\int_{-1}^{\theta_\mu}
(\theta_\mu-t)^{-\alpha}
\bigl(P_j^{(T)}\bigr)'(t)\,dt .
\]
These weights are independent of the block index \(k\), because all blocks have
the same length and the same evaluation pattern.

For a history block \(I_{k-d}\), \(d\ge 1\), the whole source block contributes.
For \(x_{k,\mu}=c_k+h\theta_\mu\) and \(\xi=c_{k-d}+ht\), we have
\[
\frac{x_{k,\mu}-c_{k-d}}{h}=2d+\theta_\mu .
\]
Equivalently, since \(\theta_\mu=2\rho_\mu-1\),
\[
\lambda_{d,\mu}=2d+2\rho_\mu-1 .
\]
Therefore the lag-\(d\) history weights are
\[
A_{d,\mu,j}^{(\alpha)}
=
\frac{h^{-\alpha}}{\Gamma(1-\alpha)}
\int_{-1}^{1}
(\lambda_{d,\mu}-t)^{-\alpha}
\bigl(P_j^{(T)}\bigr)'(t)\,dt ,
\qquad d=1,\ldots,K-1 .
\]
The important point is that \(A_{d,\mu,j}^{(\alpha)}\) depends on the target
and source blocks only through the lag \(d=k-\ell\), not through the absolute
indices \(k\) and \(\ell\). This is the source of the block-convolution
structure.

With these notations, the LLF-Caputo approximation at \(x_{k,\mu}\) can be
written as
\[
D_{N_b,\eta}^{\alpha}f(x_{k,\mu})
=
\sum_{j=0}^{N_b-1}B_{\mu,j}^{(\alpha)}c_{k,j}
+
\sum_{d=1}^{k-1}
\sum_{j=0}^{N_b-1}
A_{d,\mu,j}^{(\alpha)}c_{k-d,j}.
\]
The first term is the current-block contribution. The second term is the
history contribution. For each fixed pair \((\mu,j)\), it has the form of a
one-sided discrete convolution in the block index:
\[
\sum_{d=1}^{k-1} A_{d,\mu,j}^{(\alpha)}c_{k-d,j}.
\]
Thus the nonlocal Caputo history is converted into a collection of blockwise
lag convolutions.

The implementation consists of the following steps.

\begin{enumerate}
\item Compute the local LLF coefficients \(c_{k,j}\) on all blocks. Since the
same reference matrix is used on every block, the weighted GTSVD factorization
is formed only once and then reused.
\item Precompute the current weights \(B_{\mu,j}^{(\alpha)}\) and the lag
weights \(A_{d,\mu,j}^{(\alpha)}\). Each of these weights is evaluated by the
exact moment formulas in Section~\ref{subsec:exact_moments}.
\item For each target block \(k\), assemble the current contribution and the
history contribution using the formula above. The history part may be evaluated
directly by summation over \(d\), or by FFT-based convolution over the block
index for each fixed \((\mu,j)\).
\end{enumerate}

Let \(M\) denote the total number of unique local sampling points. In the
experiments,
\[
M=K(m-1)+1.
\]
The number of evaluation points is \(K s\), which is of the same order as \(M\)
when \(s\) is comparable to \(m\). The current weights require \(sN_b\) storage,
and the lag-dependent history weights require \((K-1)sN_b\) storage. Hence the
total storage for Caputo weights is
\[
O(KsN_b)=O(MN_b).
\]
For fixed local parameters \(s\) and \(N_b\), this is linear in the number of
samples. By contrast, a dense global differentiation matrix acting on all
evaluation and sample points would require \(O(M^2)\) storage.

If the history term is accumulated directly, the cost is
\[
O(K^2sN_b),
\]
because each target block sums over all previous source blocks. If the
lag-convolution structure is exploited by FFTs, then for each pair \((\mu,j)\)
one convolution over the block index is required. The total application cost is
therefore
\[
O(sN_bK\log K)=O(N_bM\log M).
\]
Thus, for fixed local parameters, the block-convolution implementation has
nearly linear complexity up to the logarithmic factor. The offline cost of
forming and factorizing the local weighted matrix is independent of \(K\), and
the precomputed moment weights are reused for all functions with the same
partition, evaluation pattern, and fractional order.

\begin{remark}
The complexity statements above are expressed in terms of the sample count
\(M\), since this is the most meaningful scale when comparing local equispaced
methods with global spectral or finite difference methods. The constants depend
on the fixed local parameters \(m\), \(s\), and \(N_b\), but these parameters are
small in the experiments.
\end{remark}
\section{Error Analysis}
\label{sec:error_analysis}

The analysis in this section separates the Caputo part from the local frame
reconstruction part. The proposed method does not use an ordinary Legendre
projection on \([-1,1]\). Instead, on each local reference interval
\(\Lambda=[-1,1]\), it uses the restricted Legendre frame
\[
P_j^{(T)}(t)=\frac{1}{\sqrt T}p_j\left(\frac{t}{T}\right),
\qquad t\in[-1,1],
\]
where \(P_j^{(T)}\) is orthonormal on the extended interval \([-T,T]\).
The local coefficients are therefore obtained from a regularized frame system,
rather than from a standard orthogonal projection.

The key point is that the Caputo differentiation error can be written exactly
as a fractional integral of the derivative reconstruction error. Thus the main
task is to estimate the derivative error produced by the local weighted GTSVD
reconstruction.

\subsection{Local weighted reconstruction and derivative error}
\label{subsec:local_weighted_derivative_error}

Let
\[
\mathcal V_{N_b}^{(T)}
=
\operatorname{span}
\left\{
P_0^{(T)}|_{\Lambda},
P_1^{(T)}|_{\Lambda},
\ldots,
P_{N_b-1}^{(T)}|_{\Lambda}
\right\}
\]
be the finite local Legendre frame space on \(\Lambda=[-1,1]\). For a
coefficient vector
\[
\mathbf c=(c_0,c_1,\ldots,c_{N_b-1})^T,
\]
define the synthesis operator
\[
L_{N_b}^{(T)}\mathbf c
=
\sum_{j=0}^{N_b-1}c_j P_j^{(T)}|_{\Lambda}.
\]
Let
\[
q_{N_b,\eta}
=
Q_{N_b,\eta}^{(T,R)}g
\in \mathcal V_{N_b}^{(T)}
\]
be the local weighted GTSVD reconstruction of \(g\), and denote the local
reconstruction error by
\[
e_{N_b,\eta}=g-q_{N_b,\eta}.
\]

The weighted GTSVD theory is used through the following two estimates:
\begin{equation}
\|e_{N_b,\eta}\|_{L^2(\Lambda)}
\le
\varepsilon_{N_b,\eta},
\label{eq:local_L2_error}
\end{equation}
and
\begin{equation}
\|e_{N_b,\eta}\|_{X_r(\Lambda)}
\le
M_r ,
\label{eq:local_weighted_smoothness_bound}
\end{equation}
where \(X_r(\Lambda)\) is a Hilbert-scale or coefficient-weighted smoothness
space associated with the weight matrix \(R\). In the algebraic case, one may
take \(X_r(\Lambda)=H^r(\Lambda)\). More generally, the norm
\(\|\cdot\|_{X_r(\Lambda)}\) should be understood as the function-space
counterpart of the weighted coefficient bound produced by the GTSVD
regularization.

The following interpolation estimate gives a derivative bound from
\eqref{eq:local_L2_error} and \eqref{eq:local_weighted_smoothness_bound}.

\begin{theorem}
\label{thm:local_derivative_from_weighted_gtsvd}
Let \(r>1\). Suppose that the local weighted GTSVD reconstruction error
\(e_{N_b,\eta}=g-q_{N_b,\eta}\) satisfies
\[
\|e_{N_b,\eta}\|_{L^2(\Lambda)}
\le
\varepsilon_{N_b,\eta},
\qquad
\|e_{N_b,\eta}\|_{H^r(\Lambda)}
\le
M_r .
\]
Then
\begin{equation}
\|g'-q_{N_b,\eta}'\|_{L^2(\Lambda)}
\le
C_r
M_r^{1/r}
\varepsilon_{N_b,\eta}^{(r-1)/r},
\label{eq:local_derivative_algebraic}
\end{equation}
where \(C_r\) is independent of \(\varepsilon_{N_b,\eta}\).
\end{theorem}

\proof
Since
\[
g'-q_{N_b,\eta}'=e_{N_b,\eta}',
\]
we have
\[
\|g'-q_{N_b,\eta}'\|_{L^2(\Lambda)}
\le
\|e_{N_b,\eta}\|_{H^1(\Lambda)} .
\]
By the interpolation inequality between \(L^2(\Lambda)\) and
\(H^r(\Lambda)\),
\[
\|e_{N_b,\eta}\|_{H^1(\Lambda)}
\le
C_r
\|e_{N_b,\eta}\|_{L^2(\Lambda)}^{1-1/r}
\|e_{N_b,\eta}\|_{H^r(\Lambda)}^{1/r}.
\]
Using the two assumptions gives
\[
\|g'-q_{N_b,\eta}'\|_{L^2(\Lambda)}
\le
C_r
\varepsilon_{N_b,\eta}^{1-1/r}
M_r^{1/r},
\]
which proves the result.
\proofend
\smallskip

This theorem shows that if the GTSVD reconstruction is accurate in \(L^2\)
and remains bounded in a weighted smoothness norm, then the derivative error
also converges. In particular, for algebraic weights corresponding to
\(H^r\)-regularity, the first derivative error behaves like
\[
O\left(\varepsilon_{N_b,\eta}^{(r-1)/r}\right).
\]

\subsection{Exponential weights and analytic local functions}
\label{subsec:exponential_weight_derivative_error}

For analytic functions, the coefficient decay is usually exponential. This is
the case in which the exponentially weighted GTSVD reconstruction is most
relevant.

Assume that the weighted error is controlled by an exponential weight. In
coefficient form, this means that the error expansion satisfies a bound of the
form
\begin{equation}
\sum_{j\ge0}e^{2\beta j}|\widehat e_j|^2
\le
M_\beta^2,
\label{eq:exponential_weight_bound}
\end{equation}
for some \(\beta>0\), where \(\widehat e_j\) denotes the coefficient of the
local error in a compatible polynomial or frame representation. Since the
exponential weight dominates all algebraic weights, for every \(r>0\) there
exists a constant \(C_{r,\beta}\) such that
\begin{equation}
\|e_{N_b,\eta}\|_{H^r(\Lambda)}
\le
C_{r,\beta}M_\beta .
\label{eq:exp_weight_controls_Hr}
\end{equation}
Combining this with Theorem~\ref{thm:local_derivative_from_weighted_gtsvd}
gives, for every \(r>1\),
\begin{equation}
\|g'-q_{N_b,\eta}'\|_{L^2(\Lambda)}
\le
C_{r,\beta}
M_\beta^{1/r}
\varepsilon_{N_b,\eta}^{(r-1)/r}.
\label{eq:local_derivative_exp_weight_general}
\end{equation}

If, in addition, the local \(L^2\) reconstruction error satisfies
\begin{equation}
\varepsilon_{N_b,\eta}
\le
C e^{-\sigma N_b},
\qquad \sigma>0,
\label{eq:local_L2_exponential_decay}
\end{equation}
then \eqref{eq:local_derivative_exp_weight_general} implies
\[
\|g'-q_{N_b,\eta}'\|_{L^2(\Lambda)}
\le
C_{r,\beta}
e^{-\sigma(1-1/r)N_b}.
\]
Since \(r>1\) can be chosen arbitrarily large when the exponential weighted
bound \eqref{eq:exponential_weight_bound} holds, we obtain the following
near-exponential derivative estimate.

\begin{corollary}
\label{cor:local_derivative_exponential}
Suppose that the local weighted GTSVD reconstruction satisfies
\[
\|g-q_{N_b,\eta}\|_{L^2(\Lambda)}
\le
C e^{-\sigma N_b},
\]
and that the exponential weighted bound
\eqref{eq:exponential_weight_bound} holds. Then, for every
\(0<\sigma'<\sigma\), there exists a constant \(C_{\sigma'}\), independent of
\(N_b\), such that
\begin{equation}
\|g'-q_{N_b,\eta}'\|_{L^2(\Lambda)}
\le
C_{\sigma'} e^{-\sigma' N_b}.
\label{eq:local_derivative_exponential}
\end{equation}
\end{corollary}

\proof
Choose \(r>1\) sufficiently large such that
\[
\sigma\left(1-\frac1r\right)>\sigma' .
\]
Then \eqref{eq:local_derivative_exp_weight_general} and
\eqref{eq:local_L2_exponential_decay} give
\[
\|g'-q_{N_b,\eta}'\|_{L^2(\Lambda)}
\le
C_{r,\beta}e^{-\sigma(1-1/r)N_b}
\le
C_{\sigma'}e^{-\sigma' N_b}.
\]
\proofend
\smallskip

\begin{remark}
The above result explains why the derivative error observed in the numerical
experiments decays rapidly for analytic functions. The argument does not
identify the GTSVD coefficients with ordinary Legendre projection coefficients.
Rather, it uses the two characteristic outputs of the weighted GTSVD
regularization: small \(L^2\) reconstruction error and bounded weighted
smoothness of the reconstruction error.
\end{remark}

\subsection{Transfer to the physical interval}
\label{subsec:physical_derivative_error}

We now transfer the local derivative estimate from the reference interval to
the physical interval. Let
\[
I_k=[a_{k-1},a_k],
\qquad
x=c_k+s_kt,
\qquad
t\in[-1,1],
\]
where
\[
c_k=\frac{a_{k-1}+a_k}{2},
\qquad
s_k=\frac{a_k-a_{k-1}}{2}.
\]
The local reference function is
\[
g_k(t)=f(c_k+s_kt).
\]
Let
\[
q_{k,N_b,\eta}(t)
=
Q_{N_b,\eta}^{(T,R)}g_k(t)
\]
be the local weighted GTSVD reconstruction of \(g_k\). The corresponding
physical reconstruction is
\[
p_{k,N_b,\eta}(x)
=
q_{k,N_b,\eta}
\left(\frac{x-c_k}{s_k}\right),
\qquad x\in I_k .
\]
The piecewise derivative approximation is
\[
q_{N_b}^{\eta}(x)
=
p_{k,N_b,\eta}'(x),
\qquad x\in I_k .
\]
Since
\[
p_{k,N_b,\eta}'(x)
=
\frac1{s_k}q_{k,N_b,\eta}'(t),
\qquad
g_k'(t)
=
s_k f'(c_k+s_kt),
\]
we obtain the exact scaling relation
\begin{equation}
\|f'-q_{N_b}^{\eta}\|_{L^2(I_k)}
=
s_k^{-1/2}
\|g_k'-q_{k,N_b,\eta}'\|_{L^2(\Lambda)} .
\label{eq:physical_derivative_scaling}
\end{equation}
Therefore,
\begin{equation}
\|f'-q_{N_b}^{\eta}\|_{L^2(0,b)}^2
=
\sum_{k=1}^{K}
s_k^{-1}
\|g_k'-q_{k,N_b,\eta}'\|_{L^2(\Lambda)}^2 .
\label{eq:global_derivative_scaling}
\end{equation}

Using Theorem~\ref{thm:local_derivative_from_weighted_gtsvd}, we have the
following algebraic-weight estimate.

\begin{corollary}
\label{cor:global_derivative_algebraic}
Assume that, for each local interval,
\[
\|g_k-q_{k,N_b,\eta}\|_{L^2(\Lambda)}
\le
\varepsilon_{k,N_b,\eta},
\qquad
\|g_k-q_{k,N_b,\eta}\|_{H^r(\Lambda)}
\le
M_{k,r},
\qquad r>1 .
\]
Then
\begin{equation}
\|f'-q_{N_b}^{\eta}\|_{L^2(0,b)}
\le
C_r
\left(
\sum_{k=1}^{K}
s_k^{-1}
M_{k,r}^{2/r}
\varepsilon_{k,N_b,\eta}^{2(r-1)/r}
\right)^{1/2}.
\label{eq:global_derivative_algebraic}
\end{equation}
\end{corollary}

Similarly, using Corollary~\ref{cor:local_derivative_exponential}, we obtain
the exponential-weight estimate.

\begin{corollary}
\label{cor:global_derivative_exponential}
Assume that, for each local interval,
\[
\|g_k-q_{k,N_b,\eta}\|_{L^2(\Lambda)}
\le
C_k e^{-\sigma_k N_b},
\]
and that an exponential weighted bound of the form
\eqref{eq:exponential_weight_bound} holds. Let
\[
\sigma_*=\min_{1\le k\le K}\sigma_k .
\]
Then, for every \(0<\sigma'<\sigma_*\), there exists a constant
\(C_{\sigma'}\), independent of \(N_b\), such that
\begin{equation}
\|f'-q_{N_b}^{\eta}\|_{L^2(0,b)}
\le
C_{\sigma'}
\left(
\sum_{k=1}^{K}s_k^{-1}
\right)^{1/2}
e^{-\sigma' N_b}.
\label{eq:global_derivative_exponential}
\end{equation}
\end{corollary}

\begin{remark}
The factor \(s_k^{-1/2}\) in
\eqref{eq:physical_derivative_scaling} is the usual scaling factor for first
derivatives under the affine change of variables. At the same time,
localization improves the local approximability of smooth functions. For
example, the oscillatory factor \(e^{i\omega x}\) becomes
\[
e^{i\omega(c_k+s_kt)}
=
e^{i\omega c_k}e^{i(\omega s_k)t},
\]
so the effective local frequency is \(\omega s_k\). Thus the local
coefficient decay is often much faster than that of a global expansion, which
makes the exponential weighted source condition easier to satisfy on each
subinterval.
\end{remark}

\subsection{Caputo error identity and stability estimate}
\label{subsec:caputo_error_identity}

Assume that \(f\) is absolutely continuous on \([0,b]\). From
\eqref{eq:caputo_def_intro} and \eqref{eq:caputo_approx_q},
\begin{equation}
{}^C D_{0+}^{\alpha}f(x)-\mathcal D_{N_b,\eta}^{\alpha}f(x)
=
\frac{1}{\Gamma(1-\alpha)}
\int_0^x
(x-\xi)^{-\alpha}
\left(f'(\xi)-q_{N_b}^{\eta}(\xi)\right)\,d\xi .
\label{eq:caputo_error_identity}
\end{equation}
Thus the weakly singular Caputo kernel acts on the derivative reconstruction
error.

\begin{theorem}
\label{thm:caputo_stability}
Let \(0<\alpha<1\) and let \(q_{N_b}^{\eta}\in L^2(0,b)\). Then
\begin{equation}
\left\|
{}^C D_{0+}^{\alpha}f-\mathcal D_{N_b,\eta}^{\alpha}f
\right\|_{L^2(0,b)}
\le
\frac{b^{1-\alpha}}{\Gamma(2-\alpha)}
\left\|
f'-q_{N_b}^{\eta}
\right\|_{L^2(0,b)} .
\label{eq:caputo_l2_stability}
\end{equation}
If \(f'-q_{N_b}^{\eta}\in L^\infty(0,b)\), then
\begin{equation}
\left\|
{}^C D_{0+}^{\alpha}f-\mathcal D_{N_b,\eta}^{\alpha}f
\right\|_{L^\infty(0,b)}
\le
\frac{b^{1-\alpha}}{\Gamma(2-\alpha)}
\left\|
f'-q_{N_b}^{\eta}
\right\|_{L^\infty(0,b)} .
\label{eq:caputo_linf_stability}
\end{equation}
\end{theorem}

\proof
For each \(x\in[0,b]\), \eqref{eq:caputo_error_identity} gives
\[
\left|
{}^C D_{0+}^{\alpha}f(x)-\mathcal D_{N_b,\eta}^{\alpha}f(x)
\right|
\le
\frac{1}{\Gamma(1-\alpha)}
\int_0^x
(x-\xi)^{-\alpha}
\left|
f'(\xi)-q_{N_b}^{\eta}(\xi)
\right|\,d\xi .
\]
Hence
\[
\left|
{}^C D_{0+}^{\alpha}f(x)-\mathcal D_{N_b,\eta}^{\alpha}f(x)
\right|
\le
\frac{x^{1-\alpha}}{\Gamma(2-\alpha)}
\left\|
f'-q_{N_b}^{\eta}
\right\|_{L^\infty(0,x)} ,
\]
which proves \eqref{eq:caputo_linf_stability}. The \(L^2\) estimate follows
from Young's inequality for convolution. Indeed, the kernel
\[
\kappa_\alpha(t)=\frac{t^{-\alpha}}{\Gamma(1-\alpha)},
\qquad 0<t<b,
\]
satisfies
\[
\|\kappa_\alpha\|_{L^1(0,b)}
=
\frac{b^{1-\alpha}}{\Gamma(2-\alpha)} .
\]
Therefore,
\[
\|I^{1-\alpha}v\|_{L^2(0,b)}
\le
\frac{b^{1-\alpha}}{\Gamma(2-\alpha)}
\|v\|_{L^2(0,b)} .
\]
Taking
\[
v=f'-q_{N_b}^{\eta}
\]
proves \eqref{eq:caputo_l2_stability}.
\proofend
\smallskip

\subsection{Caputo convergence estimates}
\label{subsec:caputo_convergence_estimates}

Combining Theorem~\ref{thm:caputo_stability} with
Corollary~\ref{cor:global_derivative_algebraic}, we obtain the following
result for algebraic weights.

\begin{theorem}
\label{thm:caputo_algebraic_weight_error}
Let \(0<\alpha<1\). Assume that, for each local interval,
\[
\|g_k-q_{k,N_b,\eta}\|_{L^2(\Lambda)}
\le
\varepsilon_{k,N_b,\eta},
\qquad
\|g_k-q_{k,N_b,\eta}\|_{H^r(\Lambda)}
\le
M_{k,r},
\qquad r>1 .
\]
Then
\begin{equation}
\left\|
{}^C D_{0+}^{\alpha}f-\mathcal D_{N_b,\eta}^{\alpha}f
\right\|_{L^2(0,b)}
\le
\frac{b^{1-\alpha}}{\Gamma(2-\alpha)}
C_r
\left(
\sum_{k=1}^{K}
s_k^{-1}
M_{k,r}^{2/r}
\varepsilon_{k,N_b,\eta}^{2(r-1)/r}
\right)^{1/2}.
\label{eq:caputo_algebraic_weight_error}
\end{equation}
\end{theorem}

For analytic local functions and exponential weights, the Caputo approximation
inherits exponential convergence from the local weighted GTSVD reconstruction.

\begin{theorem}
\label{thm:caputo_exponential_weight_error}
Let \(0<\alpha<1\). Assume that, for each local interval,
\[
\|g_k-q_{k,N_b,\eta}\|_{L^2(\Lambda)}
\le
C_k e^{-\sigma_k N_b},
\]
and that the exponential weighted bound
\[
\sum_{j\ge0}e^{2\beta j}|\widehat e_{k,j}|^2
\le
M_{k,\beta}^2
\]
holds for the local reconstruction error. Let
\[
\sigma_*=\min_{1\le k\le K}\sigma_k .
\]
Then, for every \(0<\sigma'<\sigma_*\), there exists a constant
\(C_{\alpha,b,\sigma'}\), independent of \(N_b\), such that
\begin{equation}
\left\|
{}^C D_{0+}^{\alpha}f-\mathcal D_{N_b,\eta}^{\alpha}f
\right\|_{L^2(0,b)}
\le
C_{\alpha,b,\sigma'}
\left(
\sum_{k=1}^{K}s_k^{-1}
\right)^{1/2}
e^{-\sigma' N_b}.
\label{eq:caputo_exponential_weight_error}
\end{equation}
\end{theorem}

\begin{remark}
The exponential convergence result should be interpreted as a consequence of
the weighted GTSVD reconstruction estimates, not as a classical projection
estimate. The computed GTSVD coefficients are not generally equal to ordinary
Legendre coefficients. The role of the exponential weight is to control the
high-index components of the reconstruction error, while the \(L^2\) estimate
controls the low-index part. The interpolation argument then transfers these
two pieces of information to the derivative and finally to the Caputo
derivative.
\end{remark}

\subsection{Consistency of the moment construction}
\label{subsec:moment_consistency}

The exact moment construction has the following consistency property.

\begin{proposition}
\label{prop:moment_consistency}
Suppose that on each interval \(I_k\), the derivative approximation
\(q_{N_b}^{\eta}\) coincides with \(f'\). Then for every evaluation point
\(x_i\),
\[
\mathcal D_{N_b,\eta}^{\alpha}f(x_i)
=
{}^{C}D_{0+}^{\alpha}f(x_i).
\]
\end{proposition}

\proof
If \(q_{N_b}^{\eta}=f'\) on each contributing interval, then
\eqref{eq:caputo_approx_q} becomes
\[
\mathcal D_{N_b,\eta}^{\alpha}f(x_i)
=
\frac{1}{\Gamma(1-\alpha)}
\int_0^{x_i}
(x_i-\xi)^{-\alpha}f'(\xi)\,d\xi ,
\]
which is exactly the Caputo derivative
\({}^{C}D_{0+}^{\alpha}f(x_i)\). The moment formulas provide an exact
evaluation of this integral when \(q_{N_b}^{\eta}\) is represented by the
restricted polynomial frame basis.
\proofend
\smallskip

\begin{remark}
Proposition~\ref{prop:moment_consistency} is a statement about the exactness
of the Caputo moment weights once the local derivative has been represented
exactly. It is not a statement that the truncated weighted GTSVD automatically
recovers all local polynomial coefficients exactly.
\end{remark}

\section{Numerical Experiments}
\label{sec:numerical_experiments}

In this section, we test the proposed LLF-Caputo method. The numerical
examples serve several purposes. First, we verify the exact moment weights and
the block assembly by using polynomial functions with closed-form Caputo
derivatives. Second, we compare the proposed method with the classical L1
scheme under comparable sampling sizes. Third, we compare it with a global
Chebyshev spectral benchmark to clarify the relation between local equispaced
approximation and global nonuniform spectral approximation. Fourth, we examine
the storage and time behavior of the block-convolution implementation.
Finally, we demonstrate that the local construction can be embedded into a
block solver for a simple fractional differential equation.

The choices of the LLF parameters \(T\), \(\gamma\), and \(N_b\) follow the
parameter study in the local Legendre frame approximation method
\cite{GongZhaoWang2026}. In that work, the extension parameter \(T\) was
shown to act mainly as a stability parameter; for the economical sampling
ratio \(\gamma=1\), the lower stable threshold is approximately
\(T_1\approx 5.6\), and \(T=6\) was recommended as a robust default. The same
study also showed that increasing \(\gamma\) mainly increases the number of
sampling points without essentially reducing the required local degree, so
\(\gamma=1\) is the most economical choice. For low-to-moderate local
frequencies, \(N_b=m=15\) gives reliable high accuracy, whereas more
oscillatory functions can be handled either by increasing the number of local
blocks or by using a larger local degree. Therefore, in the present paper we
keep \(T=6\), \(\gamma=1\), and \(N_b=m=15\) fixed in all main experiments,
and use the number of local blocks \(K\) to control the global resolution.
This choice is not tuned to individual examples, but is inherited from the
LLF approximation framework.
The truncation threshold in the weighted generalized inverse is \(10^{-14}\),
and the exponential weight parameter is \(\beta=0.15\). For the differentiation
tests, \(s=15\) evaluation points are used in each block,
\[
x_{k,\mu}=\frac{k-1+\rho_\mu}{K},
\qquad
\rho_\mu=\frac{\mu}{s},\quad \mu=1,\ldots,s .
\]
The relative discrete \(L^2\) error is defined by
\[
E_2=
\frac{
\left(\frac{1}{M_e}\sum_{i=1}^{M_e}|d_i-d_i^{\rm ref}|^2\right)^{1/2}
}{
\left(\frac{1}{M_e}\sum_{i=1}^{M_e}|d_i^{\rm ref}|^2\right)^{1/2}
},
\]
where \(M_e\) is the number of evaluation points, and \(d_i\) and
\(d_i^{\rm ref}\) denote the numerical and reference values, respectively.

Except for the moment-weight verification in Section~5.1, the remaining
accuracy experiments use the analytic test functions
\[
f_1(x)=e^x,\qquad
f_2(x)=\cos(30x),\qquad
f_3(x)=\exp(\sin(2.7\pi x)+\cos(\pi x)).
\]
For \(f_1\), the reference Caputo derivative is evaluated by
\[
{}^C D_{0+}^{\alpha}e^x
=
\sum_{n=0}^{\infty}\frac{x^{n+1-\alpha}}{\Gamma(n+2-\alpha)}.
\]
For \(f_2\) and \(f_3\), reference values are computed by a high-order
Gauss--Legendre quadrature after removing the weak singularity.
\subsection{Verification of exact moment weights}
\label{subsec:moment_weight_verification}

We first verify the exact moment weights using polynomial test functions with
closed-form Caputo derivatives. This experiment is designed to test only the
moment formulas and the block assembly, and does not use the analytic
functions employed in the subsequent accuracy experiments.

For an integer \(p\ge0\), let
\[
f_p(x)=\frac{x^{p+1}}{p+1},
\qquad
f_p'(x)=x^p .
\]
Then
\[
{}^C D_{0+}^{\alpha}f_p(x)
=
\frac{\Gamma(p+1)}{\Gamma(p+2-\alpha)}
x^{p+1-\alpha}.
\]
On each local interval, \(f_p(c_k+h_kt)\) is a polynomial in \(t\). Since
\(p+1\le N_b-1\) in the tests, this local polynomial can be represented exactly
in the restricted Legendre frame. Therefore, any discrepancy between the
moment-weight result and the closed-form Caputo derivative reflects the
accuracy of the moment formulas and the block assembly, not the local
GTSVD reconstruction error.

\begin{table}[htbp]
\centering
\caption{Closed-form verification of the exact moment weights. The test
functions are \(f_p(x)=x^{p+1}/(p+1)\), for which the Caputo derivative is
known analytically. The local frame coefficients of \(f_p\) are computed
exactly, so the errors reflect only the moment-weight construction and block
assembly.}
\label{tab:moment_weight_poly_exact}
\begin{tabular}{c c c c}
\hline
\(p\) & \(\alpha\) & Max absolute error & Relative \(L^2\) error \\
\hline
0  & 0.25 & \(2.2204\times10^{-16}\) & \(1.2673\times10^{-16}\) \\
0  & 0.50 & \(4.4409\times10^{-16}\) & \(1.8435\times10^{-16}\) \\
0  & 0.75 & \(4.4409\times10^{-16}\) & \(1.9542\times10^{-16}\) \\
2  & 0.25 & \(1.6653\times10^{-16}\) & \(2.2662\times10^{-16}\) \\
2  & 0.50 & \(4.4409\times10^{-16}\) & \(2.4205\times10^{-16}\) \\
2  & 0.75 & \(3.3307\times10^{-16}\) & \(2.0051\times10^{-16}\) \\
5  & 0.25 & \(1.6653\times10^{-16}\) & \(3.4506\times10^{-16}\) \\
5  & 0.50 & \(1.6653\times10^{-16}\) & \(2.9197\times10^{-16}\) \\
5  & 0.75 & \(4.4409\times10^{-16}\) & \(3.2731\times10^{-16}\) \\
8  & 0.25 & \(1.3878\times10^{-16}\) & \(2.9064\times10^{-16}\) \\
8  & 0.50 & \(2.2204\times10^{-16}\) & \(4.0305\times10^{-16}\) \\
8  & 0.75 & \(3.3307\times10^{-16}\) & \(3.9300\times10^{-16}\) \\
12 & 0.25 & \(9.7145\times10^{-17}\) & \(2.4557\times10^{-16}\) \\
12 & 0.50 & \(3.3307\times10^{-16}\) & \(8.1929\times10^{-16}\) \\
12 & 0.75 & \(1.1102\times10^{-15}\) & \(1.8753\times10^{-15}\) \\
\hline
\end{tabular}
\end{table}
The errors in Table~\ref{tab:moment_weight_poly_exact} remain at the level of
roundoff error for all tested polynomial degrees and fractional orders. Even
for \(p=12\), which is close to the largest exactly representable polynomial
degree under \(N_b=15\), the relative error is still below \(2\times10^{-15}\).
This confirms that the exact moment weights correctly evaluate the weakly
singular Caputo action on the local frame representation. Hence, in the
following experiments, the observed errors are dominated by the local
LLF-GTSVD reconstruction and by the comparison method, rather than by the
evaluation of the Caputo history weights.
\subsection{Comparison with the classical L1 scheme}
\label{subsec:l1_comparison}

The L1 scheme is included as a classical low-order time-stepping benchmark. For
a uniform grid \(x_n=nh\), it approximates the Caputo derivative by
\[
{}^C D_{0+}^{\alpha}f(x_n)
\approx
\frac{h^{-\alpha}}{\Gamma(2-\alpha)}
\sum_{r=0}^{n-1}a_r\bigl(f(x_{n-r})-f(x_{n-r-1})\bigr),
\qquad
a_r=(r+1)^{1-\alpha}-r^{1-\alpha}.
\]
For the LLF-Caputo method, the number of unique local sampling points is
\[
M_{\rm LLF}=K(m-1)+1,
\]
whereas the L1 scheme uses \(M_{\rm L1}=Ks+1\) grid points. Thus the sample
sizes are comparable.

\begin{table}[htbp]
\centering
\caption{Comparison between the proposed LLF-Caputo method and the classical L1 scheme with comparable sample sizes. Here \(K=16\), \(T=6\), \(\gamma=1\), and \(N_b=m=15\).}
\label{tab:llf_l1_k16}
\begin{tabular}{c c c c c c}
\hline
Function & $\alpha$ & $M_{\rm LLF}$ & $M_{\rm L1}$ & LLF-Caputo $E_2$ & L1 $E_2$ \\
\hline
$e^x$ & 0.25 & 225 & 241 & 1.3476e-13 & 1.3752e-05 \\
$\cos(30x)$ & 0.25 & 225 & 241 & 1.2029e-12 & 2.1948e-03 \\
$f_3$ & 0.25 & 225 & 241 & 1.0365e-13 & 3.9597e-04 \\
$e^x$ & 0.50 & 225 & 241 & 9.4269e-13 & 8.3853e-05 \\
$\cos(30x)$ & 0.50 & 225 & 241 & 1.0837e-13 & 9.0336e-03 \\
$f_3$ & 0.50 & 225 & 241 & 4.0700e-13 & 1.9894e-03 \\
$e^x$ & 0.75 & 225 & 241 & 4.7789e-12 & 4.3062e-04 \\
$\cos(30x)$ & 0.75 & 225 & 241 & 4.9591e-13 & 2.5465e-02 \\
$f_3$ & 0.75 & 225 & 241 & 1.9354e-12 & 7.3949e-03 \\
\hline
\end{tabular}
\end{table}

Table~\ref{tab:llf_l1_k16} shows that, for \(K=16\), the LLF-Caputo method
reaches errors between \(10^{-13}\) and \(10^{-12}\), while the L1 errors remain
between \(10^{-5}\) and \(10^{-2}\). The convergence behavior for
\(\alpha=0.5\) is shown in Fig.~\ref{fig:llf_l1_alpha05}. The LLF-Caputo method
rapidly reaches a roundoff-dominated accuracy level, whereas the L1 scheme shows
algebraic convergence.

\begin{figure}[htbp]
\centering
\includegraphics[width=0.72\textwidth]{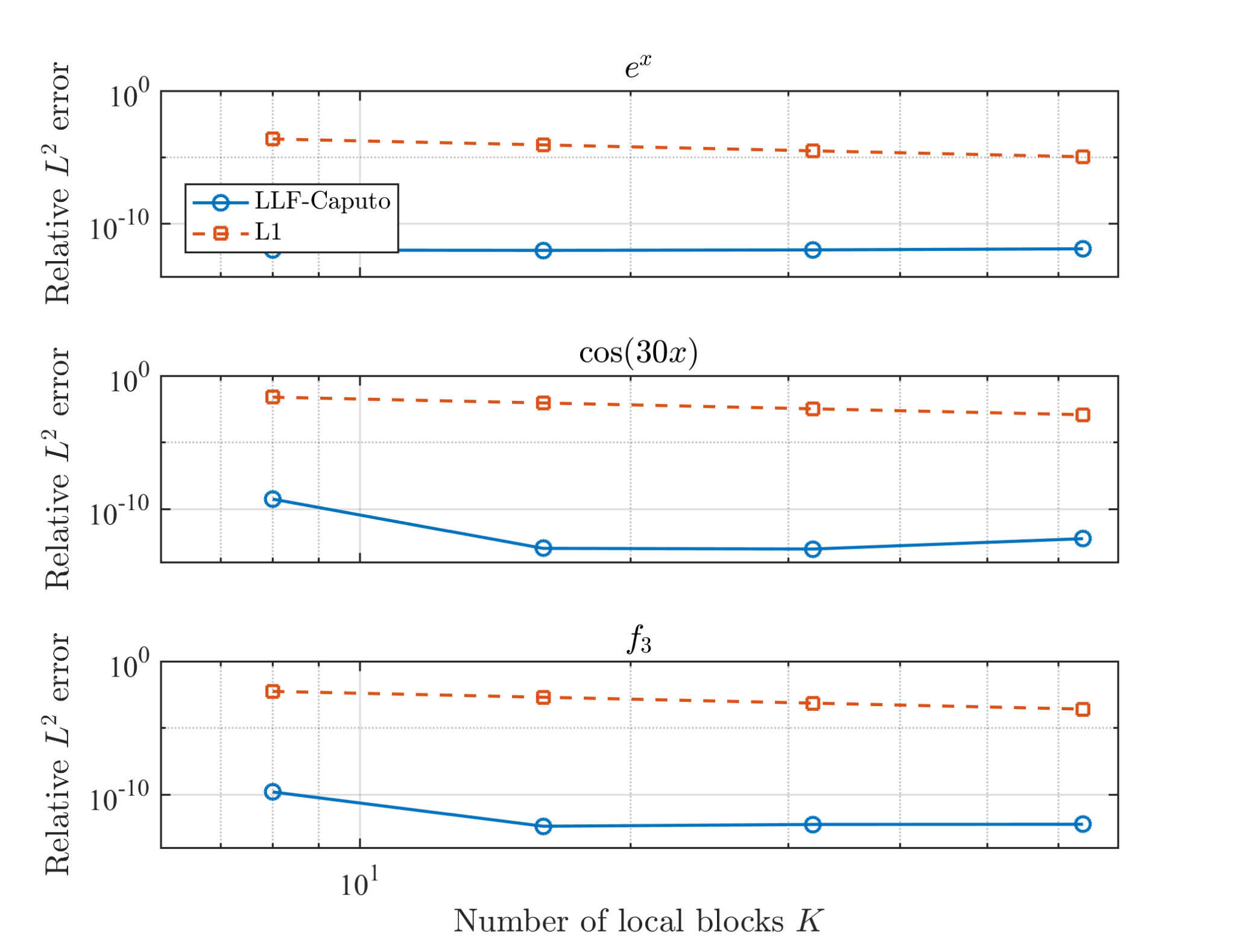}
\caption{Relative \(L^2\) errors of LLF-Caputo and the classical L1 scheme for
\(\alpha=0.5\). The three panels correspond to \(e^x\), \(\cos(30x)\), and
\(f_3\), respectively.}
\label{fig:llf_l1_alpha05}
\end{figure}

\subsection{Comparison with a global Chebyshev spectral benchmark}
\label{subsec:cheb_comparison}

We next compare the proposed method with a global Chebyshev spectral benchmark.
The Chebyshev method uses global Chebyshev--Lobatto nodes on \([0,1]\). A
global polynomial interpolant is constructed, its derivative is evaluated, and
the Caputo integral is computed using a high-order singularity-removed
Gauss--Legendre quadrature. This benchmark represents an ideal global spectral
setting with nonuniform nodes.

The purpose of this experiment is to clarify the different settings of the two
methods: the global Chebyshev method uses global nonuniform nodes, while the
proposed method uses local equispaced samples and a reusable local frame
matrix.

\begin{table}[htbp]
\centering
\caption{Comparison between the proposed LLF-Caputo method and a global Chebyshev spectral benchmark. The LLF-Caputo method uses \(K=16\), \(T=6\), \(\gamma=1\), and \(N_b=m=15\); the Chebyshev method uses global Chebyshev--Lobatto nodes.}
\label{tab:cheb_reduced_alpha05_k16}
\begin{tabular}{c c c c c c}
\hline
Function & $M_{\rm LLF}$ & LLF-Caputo $E_2$ & Factor & $M_{\rm Cheb}$ & Chebyshev $E_2$ \\
\hline
$e^x$ & 225 & 9.4269e-13 & 1 & 225 & 4.6510e-15 \\
$e^x$ & 225 & 9.4269e-13 & 2 & 113 & 2.5602e-15 \\
$e^x$ & 225 & 9.4269e-13 & 4 & 57 & 1.3299e-15 \\
$e^x$ & 225 & 9.4269e-13 & 6 & 38 & 1.1401e-15 \\
$e^x$ & 225 & 9.4269e-13 & 8 & 29 & 1.2534e-15 \\
$\cos(30x)$ & 225 & 1.0837e-13 & 1 & 225 & 3.8135e-15 \\
$\cos(30x)$ & 225 & 1.0837e-13 & 2 & 113 & 3.2492e-15 \\
$\cos(30x)$ & 225 & 1.0837e-13 & 4 & 57 & 2.4823e-15 \\
$\cos(30x)$ & 225 & 1.0837e-13 & 6 & 38 & 3.2214e-12 \\
$\cos(30x)$ & 225 & 1.0837e-13 & 8 & 29 & 1.2255e-06 \\
$f_3$ & 225 & 4.0700e-13 & 1 & 225 & 3.3173e-15 \\
$f_3$ & 225 & 4.0700e-13 & 2 & 113 & 1.6357e-15 \\
$f_3$ & 225 & 4.0700e-13 & 4 & 57 & 1.6123e-14 \\
$f_3$ & 225 & 4.0700e-13 & 6 & 38 & 6.9336e-09 \\
$f_3$ & 225 & 4.0700e-13 & 8 & 29 & 2.9708e-06 \\
\hline
\end{tabular}
\end{table}

The results show that global Chebyshev approximation is extremely effective for
analytic functions. With \(\alpha=0.5\) and \(K=16\), the Chebyshev benchmark
often reaches \(10^{-15}\)-level errors with fewer nodes than the local LLF
method. On the other hand, the LLF-Caputo method reaches high accuracy from
local equispaced samples. This distinction is important in applications where
equispaced data are natural, where local structures are present, or where a
blockwise formulation is desired.

\begin{figure}[htbp]
\centering
\includegraphics[width=0.72\textwidth]{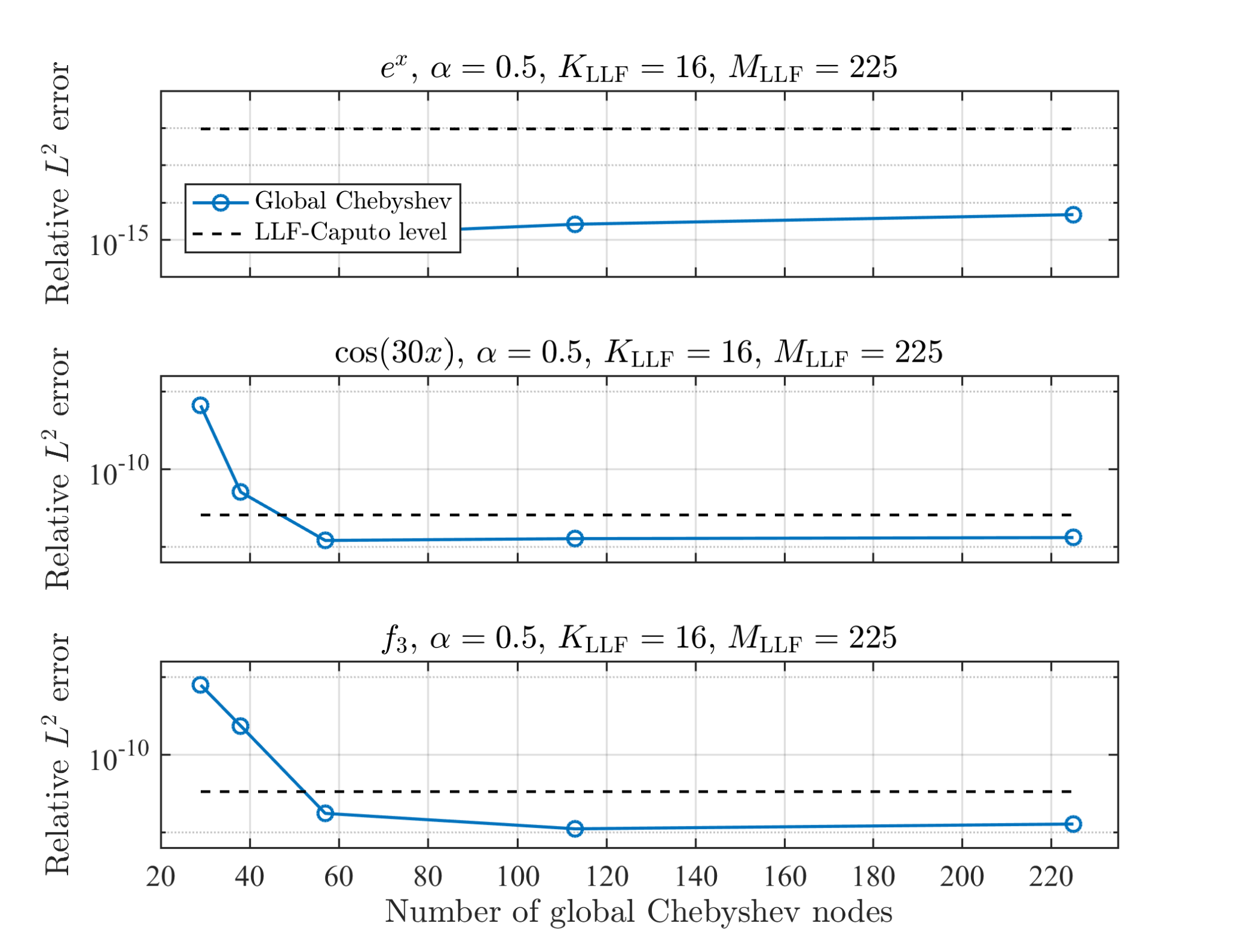}
\caption{Comparison with a global Chebyshev spectral benchmark for
\(\alpha=0.5\) and \(K_{\rm LLF}=16\). The dashed horizontal lines indicate the
LLF-Caputo error levels, while the solid curves show the global Chebyshev
errors with different numbers of Chebyshev--Lobatto nodes.}
\label{fig:cheb_benchmark_alpha05}
\end{figure}

\subsection{Computational cost}
\label{subsec:numerical_cost}

We record the computational time and storage for \(f(x)=\cos(30x)\) and
\(\alpha=0.5\). The LLF storage is the storage of the current and lag-dependent
Caputo weights. For the Chebyshev benchmark, the reported storage is a dense
global matrix proxy, included only to illustrate the cost of a direct global
matrix implementation.

\begin{table}[htbp]
\centering
\caption{Cost comparison for \(f(x)=\cos(30x)\) and \(\alpha=0.5\). The Chebyshev storage is a dense-matrix storage proxy, included only as a reference for a direct global matrix implementation.}
\label{tab:cost_alpha05_osc30}
\begin{tabular}{c c c c c c}
\hline
$K$ & $M$ & LLF storage (MB) & Cheb dense storage (MB) & LLF time (s) & Cheb time (s) \\
\hline
8 & 113 & 0.0137 & 0.0974 & 0.489 & 0.077 \\
16 & 225 & 0.0275 & 0.3862 & 0.870 & 0.377 \\
32 & 449 & 0.0549 & 1.5381 & 1.550 & 1.476 \\
64 & 897 & 0.1099 & 6.1387 & 2.397 & 5.293 \\
\hline
\end{tabular}
\end{table}

As shown in Table~\ref{tab:cost_alpha05_osc30} and
Fig.~\ref{fig:cost_alpha05_osc30}, the storage of the LLF-Caputo history weights
grows almost linearly with the number of samples, while the dense global proxy
grows quadratically. The timing results are implementation dependent, but they
show that for larger sample sizes the block-convolution LLF implementation
becomes more favorable. This supports the complexity discussion in
Section~\ref{subsec:complexity}.

\begin{figure}[htbp]
\centering
\includegraphics[width=0.70\textwidth]{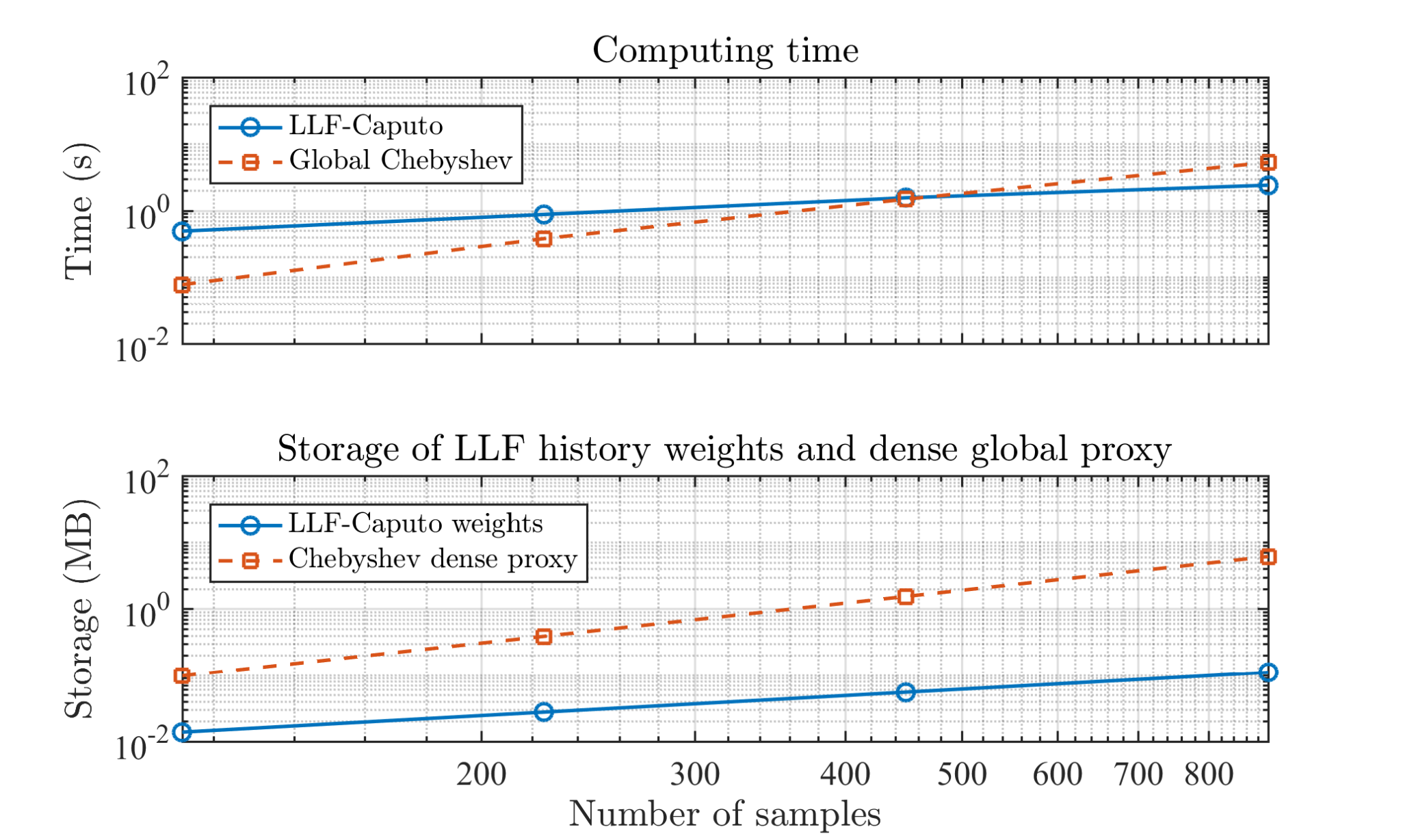}
\caption{Computing time and storage for \(f(x)=\cos(30x)\) and \(\alpha=0.5\).
The Chebyshev storage is a dense-matrix proxy and does not represent all
possible optimized global spectral implementations.}
\label{fig:cost_alpha05_osc30}
\end{figure}

\subsection{Application to a fractional differential equation}
\label{subsec:fde_example}

Finally, we consider
\[
{}^C D_{0+}^{\alpha}u(x)+\lambda u(x)=g(x),
\qquad 0<x\le 1,
\]
with \(u(0)=u_0\). We take \(\lambda=1\) and use manufactured solutions
\(u(x)=e^x\), \(u(x)=\cos(30x)\), and \(u(x)=f_3(x)\). The right-hand side is
\[
g(x)={}^C D_{0+}^{\alpha}u(x)+\lambda u(x).
\]

The local form leads to a block solver. On the \(k\)-th block, the coefficient
vector is written in the retained stable subspace as
\[
\mathbf c_k=C_{\rm red}\mathbf z_k .
\]
In the experiments, the retained rank is \(r=13\), and \(r\) collocation
equations are used on each block. The current block satisfies a small linear
system of the form
\[
\left(B^{(\alpha)}+\lambda\Phi\right)C_{\rm red}\mathbf z_k
=
\mathbf g_k-\mathbf H_k,
\]
where \(\mathbf H_k\) contains known history contributions from previous blocks.
The first row is replaced by the initial condition for \(k=1\), and by a
continuity condition at the left endpoint of the current block for \(k>1\).
Thus the computation proceeds block by block.

\begin{table}[htbp]
\centering
\caption{Numerical solution of the fractional differential equation \({}^C D_{0+}^{\alpha}u+\lambda u=g\) with manufactured solutions. Here \(K=16\), \(\lambda=1\), \(T=6\), \(\gamma=1\), \(N_b=m=15\), and the reduced local rank is \(r=13\).}
\label{tab:fde_k16}
\begin{tabular}{c c c c c c}
\hline
Exact solution & $\alpha$ & $M_{\rm LLF}$ & $M_{\rm L1}$ & LLF block solver $E_2$ & L1 $E_2$ \\
\hline
$e^x$ & 0.25 & 209 & 209 & 1.1335e-12 & 4.9336e-06 \\
$\cos(30x)$ & 0.25 & 209 & 209 & 2.4085e-12 & 2.4097e-03 \\
$f_3$ & 0.25 & 209 & 209 & 9.0698e-13 & 2.0298e-04 \\
$e^x$ & 0.50 & 209 & 209 & 4.3171e-12 & 3.2802e-05 \\
$\cos(30x)$ & 0.50 & 209 & 209 & 2.1256e-12 & 1.0277e-02 \\
$f_3$ & 0.50 & 209 & 209 & 1.5695e-12 & 1.0743e-03 \\
$e^x$ & 0.75 & 209 & 209 & 8.6387e-12 & 1.7683e-04 \\
$\cos(30x)$ & 0.75 & 209 & 209 & 6.3732e-12 & 2.9719e-02 \\
$f_3$ & 0.75 & 209 & 209 & 1.1811e-11 & 4.4728e-03 \\
\hline
\end{tabular}
\end{table}

The results in Table~\ref{tab:fde_k16} show that the LLF block solver reaches
\(10^{-12}\)-level accuracy for the manufactured solutions, while the L1 solver
remains several orders of magnitude less accurate under the same number of
unknowns. Figure~\ref{fig:fde_convergence_alpha05} shows the convergence
behavior for \(\alpha=0.5\), and Figs.~\ref{fig:fde_solution_osc30} and
\ref{fig:fde_error_osc30} give the solution and pointwise error for the
oscillatory manufactured solution \(u(x)=\cos(30x)\).

\begin{figure}[htbp]
\centering
\includegraphics[width=0.92\textwidth]{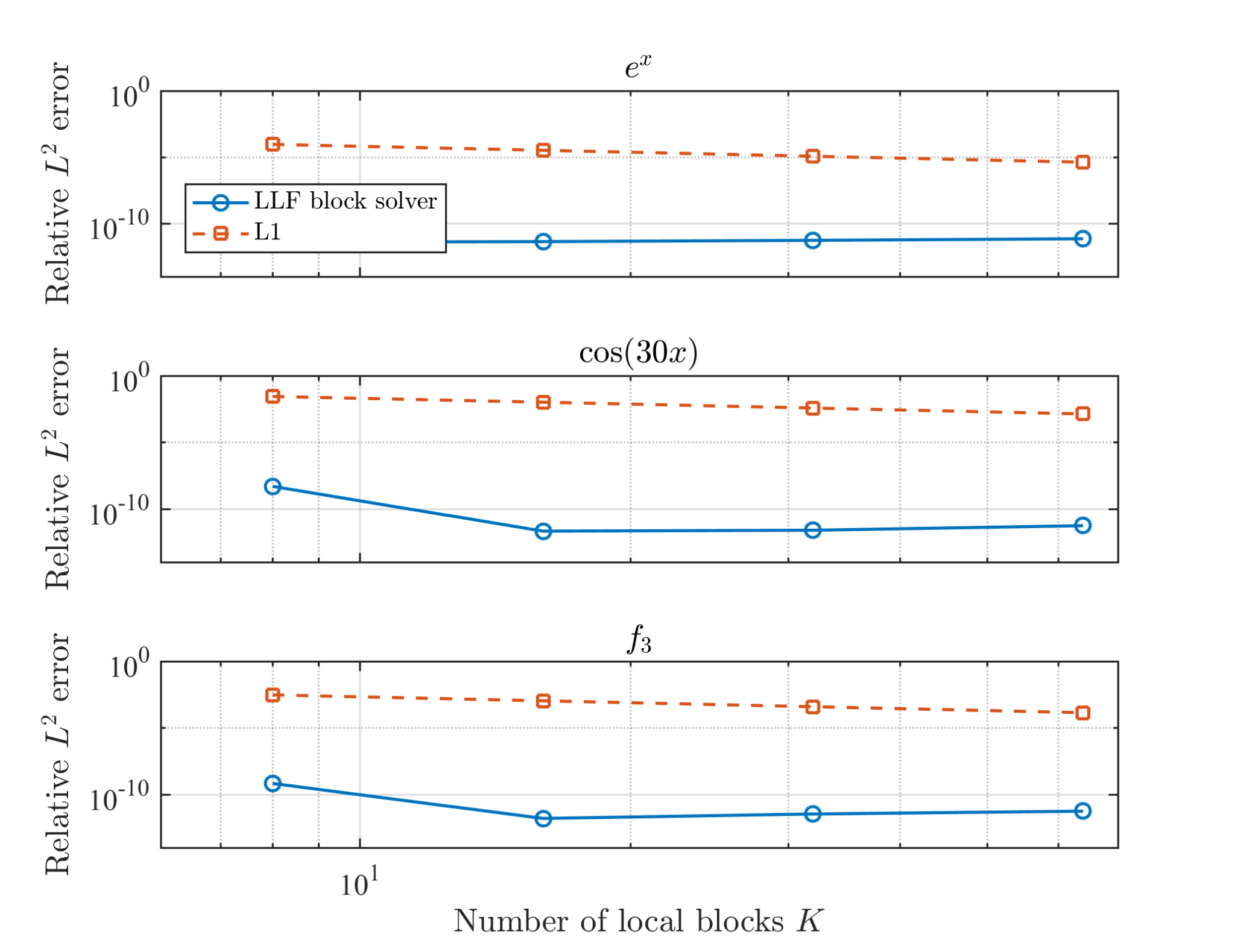}
\caption{Relative \(L^2\) solution errors for the fractional differential
equation with \(\alpha=0.5\) and \(\lambda=1\).}
\label{fig:fde_convergence_alpha05}
\end{figure}

\begin{figure}[htbp]
\centering
\includegraphics[width=0.90\textwidth]{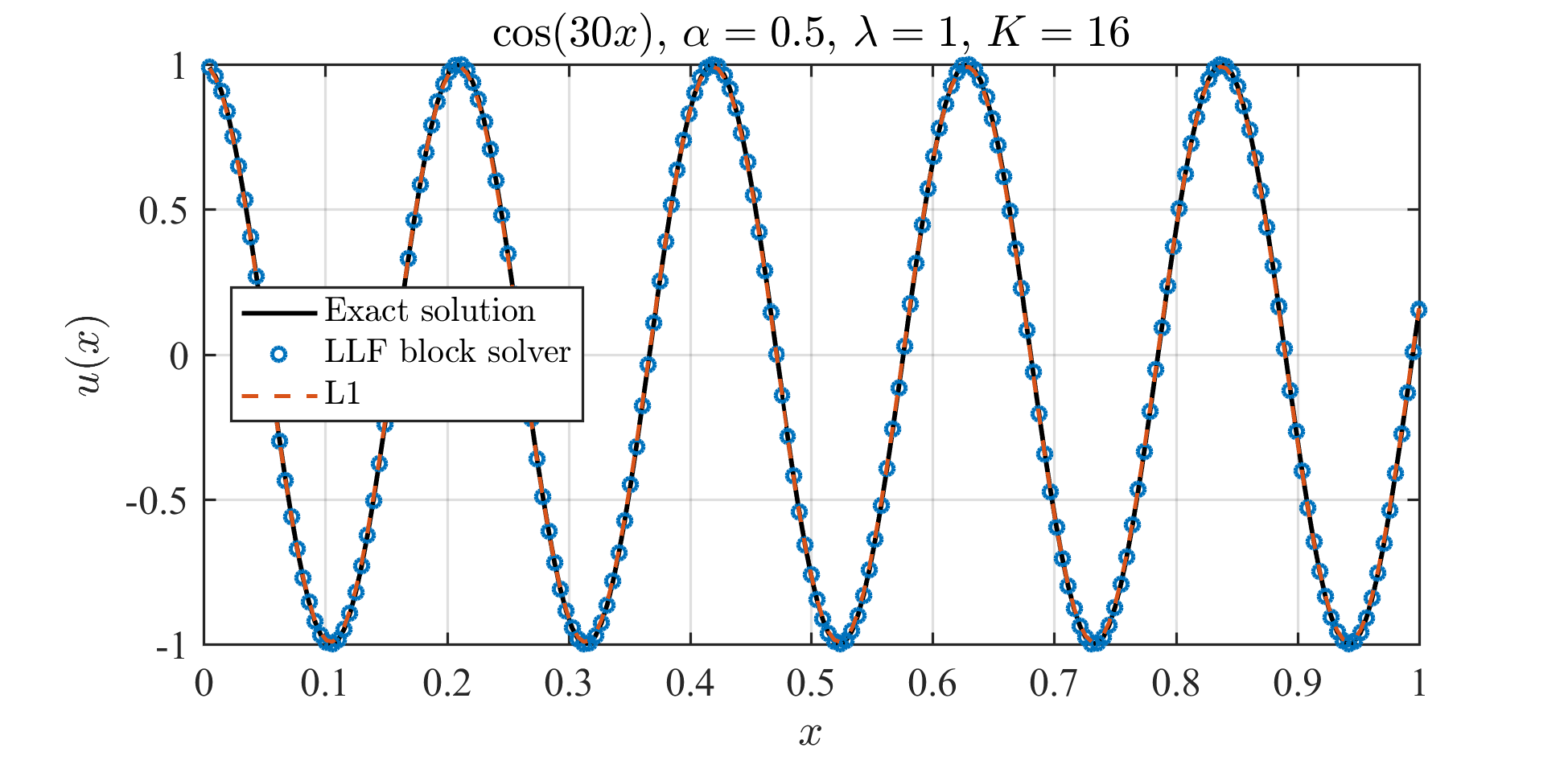}
\caption{Numerical solution of the fractional differential equation with
manufactured solution \(u(x)=\cos(30x)\), \(\alpha=0.5\), \(\lambda=1\), and
\(K=16\).}
\label{fig:fde_solution_osc30}
\end{figure}

\begin{figure}[htbp]
\centering
\includegraphics[width=0.90\textwidth]{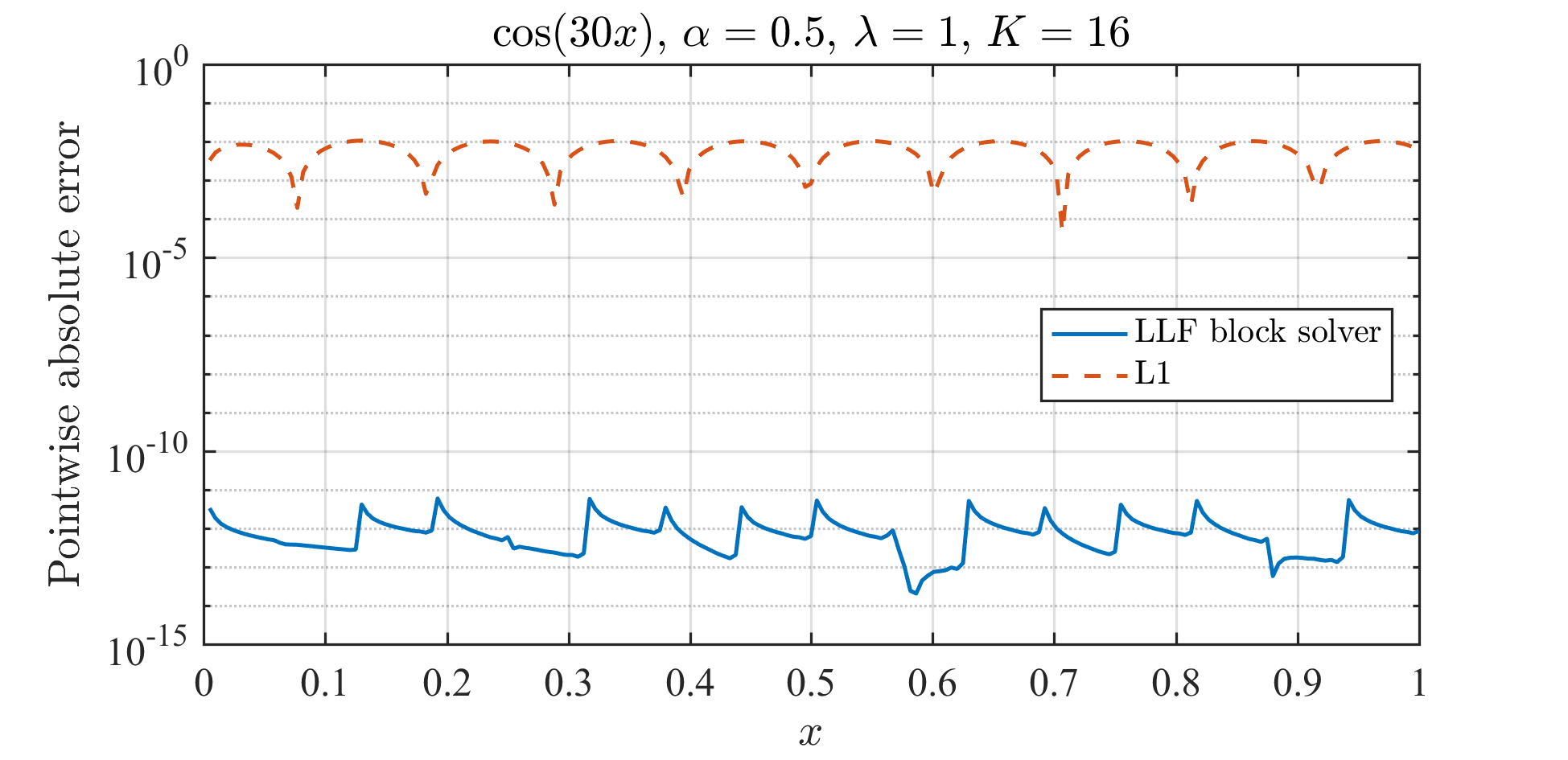}
\caption{Pointwise absolute errors for the manufactured solution
\(u(x)=\cos(30x)\), \(\alpha=0.5\), \(\lambda=1\), and \(K=16\).}
\label{fig:fde_error_osc30}
\end{figure}

This example is intended to demonstrate the block time-marching potential of
the local LLF-Caputo construction. A systematic study of nonlinear fractional
differential equations, fractional PDEs, and solutions with initial weak
singularities is left for future work.
\FloatBarrier
\section{Conclusion and Remarks}
\label{sec:conclusion}

In this paper, we developed a local Legendre frame method for the computation
of Caputo fractional derivatives. The method combines a local frame
representation from equispaced samples, an exponentially weighted GTSVD
regularization for computing stable local coefficients, and an exact moment
construction for the weakly singular Caputo weights.

The first contribution is the construction of a Caputo differentiation formula
based on restricted Legendre frames. The local basis functions are obtained by
restricting scaled Legendre polynomials from an extended interval to the
reference interval. This is different from a standard Legendre projection on
\([-1,1]\). After the local weighted GTSVD coefficients are computed, the
derivatives of the local frame basis functions are polynomials. Therefore, the
Caputo integral of each basis derivative can be evaluated by finite weighted
moments. In this way, the weak singularity of the Caputo kernel is incorporated
into precomputed analytic weights.

The second contribution is the error mechanism. The analysis shows that the
Caputo error is exactly the fractional integral of the local derivative
reconstruction error. Thus the weakly singular kernel does not introduce an
additional discretization error once the moment weights are evaluated exactly.
For the local reconstruction, the analysis does not identify the weighted
GTSVD solution with an ordinary Legendre projection. Instead, it uses the
typical outputs of weighted GTSVD regularization: an \(L^2\) reconstruction
error estimate together with a weighted smoothness bound on the reconstruction
error. Through an interpolation argument, these two estimates yield a bound
for the derivative reconstruction error. In the case of exponential weights
and analytic local functions, the derivative error retains exponential-type
convergence, and the Caputo approximation inherits the same convergence
behavior up to the bounded fractional integral operator.

The third contribution is computational. Since all local intervals share the
same reference frame matrix, the weighted GTSVD factorization is computed once
and reused on all subintervals. For uniform partitions, the Caputo history
weights depend only on the relative block distance and the local evaluation
pattern. Hence the nonlocal history term can be organized in a block form
rather than as a fully dense global differentiation matrix. When the number of
local modes is fixed, the resulting storage and computational cost scale
favorably with the total number of sampling points. This makes the method
particularly attractive when a relatively large number of equispaced samples
is needed to resolve smooth but structurally complicated functions.

The numerical experiments support these observations. The closed-form
polynomial verification confirms that the exact moment weights and the block
assembly evaluate the weakly singular Caputo action on the local frame
representation to roundoff accuracy. Tests on smooth, oscillatory, and moderately
complicated analytic functions show that high accuracy can be obtained with
fixed local parameters. The comparison with a classical L1 scheme demonstrates
the advantage of using local high-order reconstruction for Caputo
differentiation, while the comparison with a global Chebyshev benchmark
clarifies the role of the present method: global spectral methods remain
extremely effective when global nonuniform nodes are available, whereas the
proposed method provides a high-accuracy alternative based on local equispaced
data and reusable local computations. The fractional differential equation
example further shows that the local structure can be incorporated naturally
into a block time-marching formulation.

Several extensions remain for future work. First, the local nature of the
method makes it suitable for continuous piecewise smooth functions. Once
nonsmooth points are detected and aligned with local interfaces, the method
can be combined with one-sided reconstruction and local correction strategies,
which is difficult for a global spectral representation. Second, the block
structure suggests a possible route toward high-order time-marching methods
for time-fractional evolution equations. Compared with all-at-once global
spectral discretizations in time, the local frame formulation leads naturally
to a sequence of small block systems with known history contributions from
previous blocks. Third, for noisy data or solutions with initial weak
singularities, the truncation level, the exponential weight, and possible
singularity subtraction or graded partitions should be studied systematically.
These issues will be investigated in subsequent work.

\section*{Statements and Declarations}

\textbf{~~\quad Funding.}
This work was partly supported by the Natural Science Foundation of Shandong
Province under Grant Nos. ZR2026MS0015 and ZR2025MS28.

\textbf{Competing interests.}
The authors declare that they have no competing interests.

\textbf{Data availability.}
The data generated during the current study are available from the corresponding
author upon reasonable request.

\textbf{Author contributions.}
Zhenyu Zhao contributed to the conceptualization, methodology, analysis, and
writing of the manuscript. Benxue Gong contributed to the algorithm development,
numerical experiments, and manuscript revision. Tinggang Zhao and Xianzheng Jia
contributed to numerical verification, discussion, and manuscript revision. All
authors read and approved the final manuscript.


\begin{thebibliography}{99}
\normalsize

\bibitem{AdcockHuybrechs2019}
Adcock, B., Huybrechs, D.: Frames and Numerical Approximation. SIAM Review \textbf{61}(3), 443--473 (2019). doi: \href{https://doi.org/10.1137/17M1114697}{10.1137/17M1114697}.

\bibitem{AdcockHuybrechs2020}
Adcock, B., Huybrechs, D.: Frames and Numerical Approximation II: Generalized Sampling. Journal of Fourier Analysis and Applications \textbf{26}(6), 87 (2020). doi: \href{https://doi.org/10.1007/s00041-020-09793-x}{10.1007/s00041-020-09793-x}.

\bibitem{Alikhanov2015}
Alikhanov, A. A.: A New Difference Scheme for the Time Fractional Diffusion Equation. Journal of Computational Physics \textbf{280}, 424--438 (2015). doi: \href{https://doi.org/10.1016/j.jcp.2014.09.031}{10.1016/j.jcp.2014.09.031}.

\bibitem{Boyd2002}
Boyd, J. P.: A Comparison of Numerical Algorithms for Fourier Extension of the First, Second, and Third Kinds. Journal of Computational Physics \textbf{178}(1), 118--160 (2002). doi: \href{https://doi.org/10.1006/jcph.2002.7027}{10.1006/jcph.2002.7027}.

\bibitem{Diethelm2010}
Diethelm, K.: The Analysis of Fractional Differential Equations. Lecture Notes in Mathematics, vol. 2004. Springer, Berlin (2010). doi: \href{https://doi.org/10.1007/978-3-642-14574-2}{10.1007/978-3-642-14574-2}.

\bibitem{GongZhaoWang2026}
Gong, B., Zhao, Z., Wang, C.: Local Legendre Frame Approximation from Equispaced Data.
arXiv preprint arXiv:2605.09057 (2026). doi:
\href{https://doi.org/10.48550/arXiv.2605.09057}{10.48550/arXiv.2605.09057}.
.

\bibitem{Hansen1989}
Hansen, P. C.: Regularization, GSVD and Truncated GSVD. BIT Numerical Mathematics \textbf{29}, 491--504 (1989). doi: \href{https://doi.org/10.1007/BF01933211}{10.1007/BF01933211}.

\bibitem{Hansen1992}
Hansen, P. C., Sekii, T., Shibahashi, H.: The Modified Truncated SVD Method for Regularization in General Form. SIAM Journal on Scientific and Statistical Computing \textbf{13}(5), 1142--1150 (1992). doi: \href{https://doi.org/10.1137/0913066}{10.1137/0913066}.

\bibitem{Huybrechs2010}
Huybrechs, D.: On the Fourier Extension of Nonperiodic Functions. SIAM Journal on Numerical Analysis \textbf{47}(6), 4326--4355 (2010). doi: \href{https://doi.org/10.1137/090767231}{10.1137/090767231}.

\bibitem{LiZengLiu2012}
Li, C., Zeng, F., Liu, F.: Spectral Approximations to the Fractional Integral and Derivative. Fractional Calculus and Applied Analysis \textbf{15}(3), 383--406 (2012). doi: \href{https://doi.org/10.2478/s13540-012-0028-x}{10.2478/s13540-012-0028-x}.

\bibitem{LiLiu2018}
Li, L., Liu, J.-G.: A Generalized Definition of Caputo Derivatives and Its Application to Fractional ODEs. SIAM Journal on Mathematical Analysis \textbf{50}(3), 2867--2900 (2018). doi: \href{https://doi.org/10.1137/17M1160317}{10.1137/17M1160317}.

\bibitem{Lubich1986}
Lubich, C.: Discretized Fractional Calculus. SIAM Journal on Mathematical Analysis \textbf{17}(3), 704--719 (1986). doi: \href{https://doi.org/10.1137/0517050}{10.1137/0517050}.

\bibitem{Luchko2020}
Luchko, Y.: Fractional Derivatives and the Fundamental Theorem of Fractional Calculus. Fractional Calculus and Applied Analysis \textbf{23}(4), 939--966 (2020). doi: \href{https://doi.org/10.1515/fca-2020-0049}{10.1515/fca-2020-0049}.

\bibitem{MatthysenHuybrechs2016}
Matthysen, R., Huybrechs, D.: Fast Algorithms for the Computation of Fourier Extensions of Arbitrary Length. SIAM Journal on Scientific Computing \textbf{38}(2), A899--A922 (2016). doi: \href{https://doi.org/10.1137/15M1025733}{10.1137/15M1025733}.

\bibitem{MatthysenHuybrechs2018}
Matthysen, R., Huybrechs, D.: Function Approximation on Arbitrary Domains Using Fourier Extension Frames. SIAM Journal on Numerical Analysis \textbf{56}(3), 1360--1385 (2018). doi: \href{https://doi.org/10.1137/17M1134280}{10.1137/17M1134280}.

\bibitem{MokhtariMostajeran2019}
Mokhtari, R., Mostajeran, F.: A High Order Formula to Approximate the Caputo Fractional Derivative. Communications on Applied Mathematics and Computation (2019). doi: \href{https://doi.org/10.1007/s42967-019-00023-y}{10.1007/s42967-019-00023-y}.

\bibitem{Podlubny1999}
Podlubny, I.: Fractional Differential Equations. Mathematics in Science and Engineering, vol. 198. Academic Press, San Diego (1999).

\bibitem{YanSunZhang2017}
Yan, Y., Sun, Z.-Z., Zhang, J.: Fast Evaluation of the Caputo Fractional Derivative and Its Applications to Fractional Diffusion Equations: A Second-Order Scheme. Communications in Computational Physics \textbf{22}(4), 1028--1048 (2017). doi: \href{https://doi.org/10.4208/cicp.OA-2017-0019}{10.4208/cicp.OA-2017-0019}.

\bibitem{ZhaoWangLi2026}
Zhao, Z., Wang, Y., Yagola, A.G., Li, X.: A New Approach for Fourier Extension Based on Weighted Generalized Inverse.
arXiv preprint arXiv:2501.16096 (2025). doi:
\href{https://doi.org/10.48550/arXiv.2501.16096}{10.48550/arXiv.2501.16096}.

\end{thebibliography}
\end{document}